\documentclass[12pt,reqno]{article}

\usepackage{amsmath,amsthm,amsfonts,amssymb,latexsym,mathrsfs,color}

\newtheorem{thm}{Theorem}[section]
\newtheorem{lem}{Lemma}[section]
\newtheorem{prop}{Proposition}[section]
\newtheorem{conj}{Conjecture}[section]

\theoremstyle{definition}

\theoremstyle{remark}
\newtheorem{rem}{Remark}[section]

\numberwithin{equation}{section}
\newcommand{\lrf}[1]{\left\lfloor #1\right\rfloor}

\newcommand{\la}{\lambda}

\newcommand{\pngu}{G^{-}_n(v)}
\newcommand{\cngu}{G^{\circ}_n(v)}

\title{On the unimodality of independence polynomials of some graphs
\thanks{Supported partially by the National Science Foundation of China.
\newline\hspace*{5mm}
   {\it Email addresses:}\quad  wangyi@dlut.edu.cn (Y. Wang),
    zhubaoxuan@yahoo.com.cn (B.-X. Zhu)}}
\author{Yi Wang,\quad Bao-Xuan Zhu}
\date{\footnotesize School of Mathematical Sciences,
         Dalian University of Technology,
         Dalian 116024, PR China}

\begin{document}
\maketitle
\begin{abstract}
In this paper we study unimodality problems for the independence
polynomial of a graph, including unimodality, log-concavity and
reality of zeros. We establish recurrence relations and give
factorizations of independence polynomials for certain classes of
graphs. As applications we settle some unimodality conjectures and
problems.
\end{abstract}

\section{Introduction}
\hspace*{\parindent}
Throughout this paper all graphs considered are finite and simple.
Graph theoretical terms used but not defined can be found in Bondy and Murty~\cite{BM76}.

An {\it independent set} in a graph $G$ is a set of pairwise non-adjacent vertices.
A {\it maximum independent set} in $G$ is a largest independent set and its size is denoted $\alpha(G)$.
Let $i_k(G)$ denote the number of independent sets of cardinality $k$ in $G$.
Then its generating function
$$I(G;x)=\sum\limits_{k=0}^{\alpha(G)}i_k(G)x^k,\quad i_0(G)=1$$
is called the {\it independence polynomial} of $G$ (Gutman and
Harary~\cite{GH83}).
In general, it is an NP-complete problem to determine the independence
polynomial, since evaluating $\alpha(G)$ is an NP-complete
problem~(\cite{GJ79}).

Let $a_0,a_1,\ldots,a_n$ be a sequence of nonnegative numbers. It is
{\it unimodal} if there is some $m$, called a {\it mode} of the
sequence, such that
$$a_0\le a_1\le\cdots\le a_{m-1}\le a_m\ge a_{m+1}\ge\cdots\ge a_n.$$
It is {\it log-concave} if $a_k^2\ge a_{k-1}a_{k+1}$ for all $1\le k
\le n-1$.  It is {\it symmetric} if $a_k=a_{n-k}$ for $0\le k\le n$.
Clearly, a log-concave sequence of positive numbers is unimodal
(see, e.g., Brenti~\cite{Bre89}). We say that a polynomial
$\sum_{k=0}^na_kx^k$ is {\it unimodal} ({\it log-concave}, {\it
symmetric}, respectively) if the sequence of its coefficients
$a_0,a_1,\ldots,a_n$ is unimodal (log-concave, symmetric,
respectively). A mode of the sequence $a_0,a_1,\ldots,a_n$ is also
called a mode of the polynomial $\sum_{k=0}^na_kx^k$. Unimodality
problems arise naturally in many branches of mathematics and have
been extensively investigated. See Stanley's survey
article~\cite{Sta89} and Brenti's supplement~\cite{Bre94} for
details. A basic approach to unimodality problems is to use Newton's
inequalities: Let $a_0,a_1,\ldots,a_n$ be a sequence of nonnegative
numbers. Suppose that the polynomial $\sum_{k=0}^{n}a_kx^k$ has only
real zeros. Then
$$a_k^2\ge a_{k-1}a_{k+1}\left(1+\frac{1}{k}\right)\left(1+\frac{1}{n-k}\right),\quad k=1,2,\ldots,n-1,$$
and the sequence is therefore log-concave and unimodal (see Hardy,
Littlewood and P\'olya~\cite[p. 104]{HLP52}).

Unimodality problems of graph polynomials have always been of great
interest to researchers in graph theory. For example, it is
conjectured that the chromatic polynomial of a graph is
unimodal~(Read~\cite[p. 68]{Rea68}) and even
log-concave~(Welsh~\cite[p. 266]{Wel76}). It is also well known that
the matching polynomial of a graph has only real zeros~\cite{HL72,
Sch81}. There has been an extensive literature in recent years on
the unimodality problems of independence polynomials (see
\cite{AMSE87,BHN04,BN05,CS07,Ham90,HL94,LM02,LM04CJM,LM04WSEAS,LM05,LM06CN,LM06EJC,LM03,LM07,Man09,Zhu07}
for instance). The independence polynomial is closely related to the
matching polynomial. Actually, the matching polynomial
of a graph $G$ coincides with the independence polynomial of the
line graph of $G$. In other words, the independence polynomial can
be regarded as a generalization of the matching polynomial. Wilf
asked whether the independence polynomials share the same
unimodality property as the matching polynomials. However Alavi,
Malde, Schwenk, Erd\H{o}s~\cite{AMSE87} provided examples to
demonstrate that independence polynomials are not unimodal in
general. They proposed the conjecture that the independence
polynomial of every tree or forest is unimodal. It is also conjectured
that the independence polynomial of every very well-covered graph is
unimodal (see Levit and Mandrescu~\cite{LM03} for details).
Nevertheless, the independence polynomials for certain special classes of graphs are unimodal
and even have only real zeros.
For example, the independence polynomial of a line graph has only real zeros.
Chudnovsky and Seymour~\cite{CS07} further showed that the independence polynomial of a claw-free graph has only real zeros
(the line graph of any graph is claw-free).
Brown and Nowakowski~\cite{BN05} showed that
while the independence polynomial of almost every graph of order $n$ has a nonreal zero,
the average independence polynomials always have all real and simple zeros.
Brown {\it et al.}~\cite{BDN00} showed that
for every well-covered graph $G$ there is a well-covered graph $H$,
with $\alpha(G)=\alpha(H)$, such that $G$ is an induced subgraph of $H$ and
$I(H; x)$ has all real and simple zeros.
Mandrescu~\cite{Man09} showed that given a graph $G$ whose $I(G; x)$ has only real zeros,
there are infinitely many graphs $H$ having $G$ as a subgraph, such that $I(H; x)$ has only real zeros.

The present paper is devoted to the study of unimodality problems
for the independence polynomial of a graph, including unimodality,
log-concavity and reality of zeros. We establish recurrence
relations and give factorizations of independence polynomials for
certain classes of graphs. As applications we settle some
unimodality conjectures and problems in the literature.
\section{Lemmas}
\hspace*{\parindent}
To explain our approach and prove the theorems, we state some simple
lemmas without proofs in this section.

Let $G=(V,E)$ be a simple graph and $v\in V$. Denote $N(v)=\{w:
\text{$w\in V$ and $vw\in E$}\}$ and $N[v]=N(v)\cup \{v\}$. The
following equalities are fundamental in calculating of the
independence polynomials and will be repeatedly used in the sequel.
\begin{lem}[\cite{GH83,HL94}]\label{lem-IP}
Let $G=(V,E)$ be a simple graph. Then
\begin{itemize}
  \item [\rm (i)] $I(G;x)=I(G-v;x)+xI(G-N[v];x)$ for arbitrary $v\in V$;
  \item [\rm (ii)] $I(G;x)=I(G-e;x)-x^2I(G-N(u)\bigcup N(v);x)$ for arbitrary $e=uv\in
  E$.
\end{itemize}
\end{lem}

Lemma~\ref{lem-IP} can be used to establish the recurrence relations
for independence polynomials for some classes of graphs. For
example, let $P_n$ be a path with $n$ vertices. Then $I(P_0;x)=1,
I(P_1;x)=1+x$ and
\begin{equation}\label{rr-path}
I(P_n;x)=I(P_{n-1};x)+xI(P_{n-2};x),\quad n=2,3,4,\ldots
\end{equation}
by Lemma~\ref{lem-IP}~(i). For such a sequence of polynomials
satisfying certain recurrence relation, there are various methods
for showing the sequence has only real zeros. The following result
is a special case of Corollary 2.4 in Liu and Wang~\cite{LW07}.
\begin{lem}\label{lem-LW}
Let $\{Q_n(x)\}_{n\ge 0}$ be a sequence of polynomials with
nonnegative coefficients such that
\begin{itemize}
\item [\rm (i)]
$Q_{n}(x)=a_n(x)Q_{n-1}(x)+b_n(x)Q_{n-2}(x)$ for $n\ge 2$;
\item [\rm (ii)]
$Q_0(x)$ is a constant and $\deg Q_{n-1}\le\deg Q_n\le\deg
Q_{n-1}+1$.
\end{itemize}
If $b_n(x)\le 0$ whenever $x\le 0$, then $\{Q_n(x)\}$ has only real
zeros. Furthermore, the zeros of $Q_n(x)$ are separated by that of
$Q_{n-1}(x)$.
\end{lem}

From Lemma~\ref{lem-LW} it immediately follows a well-known result
that $I(P_n;x)$ has only real zeros, and is therefore log-concave
and unimodal by Newton's inequalities. However, Lemma~\ref{lem-LW}
is useless in showing the unimodality and log-concavity of
polynomials when the condition (ii) is not satisfied. Hence it is
necessary to use other proof techniques for attacking unimodality
problems on sequences of polynomials.

The following two lemmas are elementary but play the key role in our
approach.
\begin{lem}[\cite{Bru92}]\label{lem-BF}
Let $\{z_n\}_{n\ge 0}$ be a sequence satisfying the linear
recurrence relation
$$z_n=az_{n-1}+bz_{n-2},\quad n=2,3,\ldots.$$
If $a^2+4b>0$, then the closed form for the sequence is
$$z_n=\frac{(z_1-z_0\la_2)\la_1^n+(z_0\la_1-z_1)\la_2^n}{\la_1-\la_2},\quad n=0,1,2,\ldots,$$
where
$$\la_1=\frac{a+\sqrt{a^2+4b}}{2},\quad
\la_2=\frac{a-\sqrt{a^2+4b}}{2}$$ are the roots of quadratic
equation $\la^2-a\la-b=0$.
\end{lem}
\begin{rem}
It is clear that $\la_1+\la_2=a$ and $\la_1\la_2=-b$.
\end{rem}
\begin{lem}[\cite{BC55}]\label{lem-decomp}
Let $\la_1,\la_2\in\mathbb{R}$ and $n\in\mathbb{N}$.
\begin{itemize}
\item[\rm (i)]
If $n$ is odd, then
$\la_1^n-\la_2^n=(\la_1-\la_2)\prod_{s=1}^{\frac{n-1}{2}}\left[(\la_1+\la_2)^2-4\la_1\la_2\cos^2\frac{s\pi}{n}\right]$.
\item[\rm (ii)]
If $n$ is even, then
$\la_1^n-\la_2^n=(\la_1-\la_2)(\la_1+\la_2)\prod_{s=1}^{\frac{n-2}{2}}\left[(\la_1+\la_2)^2-4\la_1\la_2\cos^2\frac{s\pi}{n}\right]$.
\item[\rm (iii)]
If $n$ is odd, then
$\la_1^n+\la_2^n=(\la_1+\la_2)\prod_{s=1}^{\frac{n-1}{2}}\left[(\la_1+\la_2)^2-4\la_1\la_2\cos^2\frac{(2s-1)\pi}{2n}\right]$.
\item[\rm (iv)]
If $n$ is even, then
$\la_1^n+\la_2^n=\prod_{s=1}^{\frac{n}{2}}\left[(\la_1+\la_2)^2-4\la_1\la_2\cos^2\frac{(2s-1)\pi}{2n}\right]$.
\end{itemize}
\end{lem}
\begin{rem}\label{divide}
If $\la_1\la_2\ge 0$, then
\begin{equation}\label{d-decomp}\frac{\la_1^n-\la_2^n}{\la_1-\la_2}
=\prod_{s=1}^{n-1}\left(\la_1+\la_2-2\sqrt{\la_1\la_2}\cos\frac{s\pi}{n}\right)
=\prod_{s=1}^{n-1}\left(\la_1+\la_2+2\sqrt{\la_1\la_2}\cos\frac{s\pi}{n}\right).
\end{equation}
\end{rem}

Applying Lemmas~\ref{lem-BF} and \ref{lem-decomp} to the recursion
(\ref{rr-path}) we may obtain the factorization of independence
polynomials for paths
\begin{equation}\label{exp-path}
I(P_n;x)=\prod_{s=1}^{\lrf{(n+1)/2}}\left(1+4x\cos^2{\frac{s\pi}{n+2}}\right),
\end{equation}
where $\lrf{\cdot}$ is the floor function (see Remark~\ref{rem-p}).

The independence polynomial of a disconnected graph is the product
of the independence polynomials of its connected components. So the
following result will be very useful in solving the unimodality
problems for independence polynomials. We refer the reader to
Stanley's survey article~\cite{Sta89} for further information.
\begin{lem}\label{product}
Let $f(x)$ and $g(x)$ be polynomials with positive coefficients.
\begin{itemize}
\item [\rm (i)] If both $f(x)$ and $g(x)$ have only real zeros, then
so does their product $f(x)g(x)$.
\item [\rm (ii)] If both $f(x)$ and $g(x)$ are log-concave, then
so is their product $f(x)g(x)$.
\item [\rm (iii)] If $f(x)$ is log-concave and $g(x)$ is unimodal, then
their product $f(x)g(x)$ is unimodal.
\item [\rm (iv)] If both $f(x)$ and $g(x)$ are symmetric and unimodal, then
so is their product $f(x)g(x)$.
\end{itemize}
\end{lem}
\section{Main Results and Applications}
\hspace*{\parindent}
We present our main results in this section. As applications we
settle certain conjectures and problems that have occurred in the
literature.
\subsection{Concatenation of Graphs}
\hspace*{\parindent}
Let $G$ be a simple graph and $v$ a vertex of $G$. Let $\pngu$
denote the graph obtained by identifying (gluing) each vertex of the
path $P_n$ with the vertex $v$ of $n$ copies $G$ respectively.
We call $\pngu$ the $n$-{\it concatenation} of the graph $G$ on the
vertex $v$. For example, the vertebrated graph $V_n^{(m)}$ defined
in~\cite{Zhu07} is the $n$-concatenation of the star graph $K_{1,m}$
on the center $v$, and the $(n,m)$-firecracker graph $F_n^{(m)}$
defined in~\cite{CLY97} is the $n$-concatenation of the star graph
$K_{1,m}$ on a leaf $v$ (see Figure 1).

\begin{center}
\setlength{\unitlength}{1cm}
\begin{picture}(22,2.5)(-0.5,-2.0)
\thicklines \put(1,0){\line(1,0){2.1}}
 \thicklines\put(4.05,0){\line(1,0){2}}
\put(1,0){\circle*{0.2}}\put(2.7,0){\circle*{0.2}}

\put(3.3,0){\circle*{0.1}}\put(3.6,0){\circle*{0.1}}
\put(3.9,0){\circle*{0.1}}
\put(4.4,0){\circle*{0.2}}\put(6.1,0){\circle*{0.2}}
 \put(1.1,0.2){$v_1$} \put(2.8,0.2){$v_2$}
\put(4.5,0.2){$v_{n-1}$} \put(6.2,0.2){$v_n$}
 \put(1,0){\line(0,1){0.8}}\put(1,0){\line(-1,-1){0.6}}\put(1,0){\line(1,-1){0.6}}
 \put(2.7,0){\line(0,1){0.8}}\put(2.7,0){\line(-1,-1){0.6}}\put(2.7,0){\line(1,-1){0.6}}
 \put(4.4,0){\line(0,1){0.8}}\put(4.4,0){\line(-1,-1){0.6}}\put(4.4,0){\line(1,-1){0.6}}
 \put(6.1,0){\line(0,1){0.8}}\put(6.1,0){\line(-1,-1){0.6}}\put(6.1,0){\line(1,-1){0.6}}
\put(1,0.8){\circle*{0.2}}\put(2.7,0.8){\circle*{0.2}}
\put(4.4,0.8){\circle*{0.2}}\put(6.1,0.8){\circle*{0.2}}
\put(0.4,-0.6){\circle*{0.2}}\put(1.6,-0.6){\circle*{0.2}}
\put(2.1,-0.6){\circle*{0.2}}\put(3.3,-0.6){\circle*{0.2}}
\put(3.8,-0.6){\circle*{0.2}}\put(5.0,-0.6){\circle*{0.2}}
\put(5.5,-0.6){\circle*{0.2}}\put(6.7,-0.6){\circle*{0.2}}
\put(3.5,-1.5){$V_n^{(3)}$}
\thicklines \put(9,0.8){\line(1,0){2.1}}
 \thicklines\put(12.05,0.8){\line(1,0){2}}
\put(9,0){\circle*{0.2}}\put(10.7,0){\circle*{0.2}}

\put(11.3,0.8){\circle*{0.1}}\put(11.6,0.8){\circle*{0.1}}
\put(11.9,0.8){\circle*{0.1}}
\put(12.4,0){\circle*{0.2}}\put(14.1,0){\circle*{0.2}}
 \put(9.0,1.0){$v_1$} \put(10.7,1.0){$v_2$}
\put(12.4,1.0){$v_{n-1}$} \put(14.1,1.0){$v_n$}
 \put(9,0){\line(0,1){0.8}}\put(9,0){\line(-1,-1){0.6}}\put(9,0){\line(1,-1){0.6}}
 \put(10.7,0){\line(0,1){0.8}}\put(10.7,0){\line(-1,-1){0.6}}\put(10.7,0){\line(1,-1){0.6}}
 \put(12.4,0){\line(0,1){0.8}}\put(12.4,0){\line(-1,-1){0.6}}\put(12.4,0){\line(1,-1){0.6}}
 \put(14.1,0){\line(0,1){0.8}}\put(14.1,0){\line(-1,-1){0.6}}\put(14.1,0){\line(1,-1){0.6}}
\put(9,0.8){\circle*{0.2}}\put(10.7,0.8){\circle*{0.2}}
\put(12.4,0.8){\circle*{0.2}}\put(14.1,0.8){\circle*{0.2}}
\put(8.4,-0.6){\circle*{0.2}}\put(9.6,-0.6){\circle*{0.2}}
\put(10.1,-0.6){\circle*{0.2}}\put(11.3,-0.6){\circle*{0.2}}
\put(11.8,-0.6){\circle*{0.2}}\put(13.0,-0.6){\circle*{0.2}}
\put(13.5,-0.6){\circle*{0.2}}\put(14.7,-0.6){\circle*{0.2}}
 \put(11.5,-1.5){$F_n^{(3)}$}
\end{picture}
Figure~1. A vertebrated graph and a firecracker graph.
\end{center}
\begin{thm}\label{thm-PGu}
Let $\pngu$ be the $n$-concatenation of the graph $G$ on the vertex
$v$. Then
\begin{equation}\label{exp-PGu}
I(\pngu;x)=I^{\lfloor\frac{n}{2}\rfloor}(G-v;x)
\prod_{s=1}^{\lfloor\frac{n+1}{2}\rfloor}
\left[I(G-v;x)+4xI(G-N[v];x)\cos^2\frac{s\pi}{n+2}\right].
\end{equation}
\end{thm}
\begin{proof}
Let $f_n=I(\pngu;x)$. Denote $a=I(G-v;x)$ and $b=xI(G-N[v];x)$. Then
$f_0=1,f_1=a+b$ and
\begin{equation}\label{rr-f}
f_n=af_{n-1}+abf_{n-2}
\end{equation}
by Lemma~\ref{lem-IP}~(i). Solve this recursion by
Lemma~\ref{lem-BF} to obtain
\begin{equation}\label{BF-f}
f_n=\frac{(a+b-\la_2)\la_1^n+(\la_1-a-b)\la_2^n}{\la_1-\la_2},
\end{equation}
where $\la_1$ and $\la_2$ are the roots of the quadratic equation
$\la^2-a\la-ab=0$. Note that $a+b-\la_2=\la_1+b=\la_1^2/a$ and
$\la_1-a-b=-\la_2^2/a$. Hence
$$f_n=\frac{{\la_1}^{n+2}-{\la_2}^{n+2}}{a(\la_1-\la_2)}.$$
Thus we have by Lemma~\ref{lem-decomp}
\begin{equation}\label{odd-f}
    f_n =
    \frac{(\la_1-\la_2)\prod\limits_{s=1}^{\frac{n+1}{2}}\left[(\la_1+\la_2)^2-4\la_1\la_2\cos^2\frac{s\pi}{n+2}\right]}{a(\la_1-\la_2)}
       =a^{\frac{n-1}{2}}\prod_{s=1}^{\frac{n+1}{2}}\left(a+4b\cos^2{\frac{s\pi}{n+2}}\right)
\end{equation}
for odd $n$, and
\begin{equation}\label{even-f}
    f_n=
    \prod\limits_{s=1}^{\frac{n}{2}}\left[(\la_1+\la_2)^2-4\la_1\la_2\cos^2\frac{s\pi}{n+2}\right]
   = a^{\frac{n}{2}}\prod_{s=1}^{\frac{n}{2}}\left(a+4b\cos^2{\frac{s\pi}{n+2}}\right)
\end{equation}
for even $n$. Combine (\ref{odd-f}) and (\ref{even-f}) to obtain
(\ref{exp-PGu}).
\end{proof}

\begin{rem}\label{rem-p}
If $G$ is the singleton graph with the unique vertex $v$, then
$\pngu$ is just the path $P_n$. Thus (\ref{exp-path}) is a
special case of (\ref{exp-PGu}).
\end{rem}

Theorem \ref{thm-PGu} gives the factorization of the independence
polynomial of $\pngu$. In what follows we present some applications
of this result.
\subsubsection{Vertebrated Graphs and Firecracker Graphs}
\hspace*{\parindent}
Let $G$ be the star graph $K_{1,m}$ with the center $v$. Then
$\pngu$ is the vertebrated graph $V_{n}^{(m)}$. Levit and
Mandrescu~\cite{LM02} showed that the independence polynomial of the
$n$-centipede $V_{n}^{(1)}$ is unimodal and further conjectured that
$I(V_{n}^{(1)};x)$ has only real zeros. Zhu~\cite{Zhu07} settled the
conjecture and showed that the independence polynomial of the
caterpillar graph $V_{n}^{(2)}$ is symmetric and unimodal. She also
proposed the following.
\begin{conj}[\cite{Zhu07}]\label{conj-zhu}
For $n,m\ge 0$, the independence polynomial $I(V_{n}^{(m)};x)$ is
unimodal.
\end{conj}

We next give a general result about the unimodality of the
independence polynomials of the vertebrated graphs, which in
particular, gives an affirmative answer to
Conjecture~\ref{conj-zhu}.
\begin{prop}\label{vg}
Let $n\ge 1$ and $m\ge 0$. Then
\begin{itemize}
\item [\rm (i)]
the independence polynomial of the vertebrated graph is
\begin{equation}\label{exp-v}
I(V_{n}^{(m)};x)=(1+x)^{m\lrf{\frac{n}{2}}}
\prod_{s=1}^{\lrf{\frac{n+1}{2}}}\left[(1+x)^{m}+4x\cos^2\frac{s\pi}{n+2}\right].
\end{equation}
\item [\rm (ii)]
$I(V_{n}^{(m)};x)$ is log-concave and therefore unimodal. In
particular, $I(V_{n}^{(m)};x)$ has only real zeros for $m=0,1,2$.
\end{itemize}
\end{prop}
\begin{proof}
(i)\quad Note that independence polynomials of the empty graph on
$m$ vertices and the null graph are $(1+x)^m$ and $1$ respectively.
Hence (\ref{exp-v}) follows immediately from (\ref{exp-PGu}).

(ii)\quad To show the log-concavity of $I(V_{n}^{(m)};x)$, it
suffices to show that each factor on the right of (\ref{exp-v}) is
log-concave by Lemma~\ref{product}~(ii).

Let $0\le a\le 1$ and $m\ge 0$. We claim that the polynomial
$$(1+x)^{m}+4xa=1+(m+4a)x+\frac{m(m-1)}{2}x^2+\frac{m(m-1)(m-2)}{6}x^3+\cdots+x^m$$
is log-concave. Actually, it suffices to prove the inequality
$$\left[\frac{m(m-1)}{2}\right]^2\ge (m+4a)\left[\frac{m(m-1)(m-2)}{6}\right]$$
since $(1+x)^{m}$ is log-concave. Clearly, it suffices to prove the
inequality when $a=1$.  In this case, the inequality is equivalent
to $m^2-7m+16\ge 0$, which is obviously true.

Thus $I(V_{n}^{(m)};x)$ is log-concave and therefore unimodal for
$m\ge 0$.

It is obvious from (\ref{exp-v}) that $I(V_{n}^{(m)};x)$ has only
real zeros for $m=0,1$. It is also easy to confirm that
$(x+1)^2+4ax=x^2+2(1+2a)x+1$ has only real zeros for $a\ge 0$. Hence
$I(V_{n}^{(2)};x)$ has only real zeros. This completes the proof of
(ii).
\end{proof}

Similarly, applying Theorem \ref{thm-PGu} to the firecracker graph,
the following result is immediate.
\begin{prop}\label{fg}
Let $n\ge 1$ and $m\ge 0$. Then
\begin{itemize}
\item [\rm (i)] $I(F_{n}^{(m)};x)=\left[(x+1)^{m}+x\right]^{\lrf{\frac{n}{2}}}
\prod_{s=1}^{\lrf{\frac{n+1}{2}}}\left[(x+1)^{m}+x+4x(x+1)^{m}\cos^2\frac{s\pi}{n+2}\right]$;
\item [\rm (ii)]
$I(F_{n}^{(m)};x)$ is log-concave and unimodal.
\end{itemize}
\end{prop}
\subsubsection{Claw-free Graphs}
\hspace*{\parindent}
Recently, Chudnovsky and Seymour~\cite{CS07} showed that the
independence polynomial of a claw-free graph has only real zeros. We
say that two real polynomials $f(x)$ and $g(x)$ are {\it compatible}
if $af(x)+bg(x)$ has only real zeros for all $a,b\ge 0$. Chudnovsky
and Seymour in fact obtained the following result.
\begin{lem}[\cite{CS07}]\label{comp}
Let $G$ be a claw-free graph. Then
\begin{itemize}
\item [\rm (i)]
$I(G-v;x)$ and $xI(G-N[v];x)$ are compatible for any $v\in V(G)$;
\item [\rm (ii)]
$I(G;x)$ has only real zeros.
\end{itemize}
\end{lem}

We can extend this result as follows.
\begin{prop}\label{cl}
If $G$ is a claw-free graph, then for any vertex $v$ of $G$, the
independence polynomial $I(\pngu;x)$ has only real zeros.
\end{prop}
\begin{proof}
Recall that $$I(\pngu;x)=I^{\lfloor\frac{n}{2}\rfloor}(G-v;x)
\prod_{s=1}^{\lfloor\frac{n+1}{2}\rfloor}
\left[I(G-v;x)+4xI(G-N[v];x)\cos^2\frac{s\pi}{n+2}\right].$$ By the
assumption, the graph $G$ is claw-free, so are the induced graphs
$G-v$ and $G-N[v]$. By Lemma~\ref{comp}, $I(G-v;x)$ and
$xI(G-N[v];x)$ have only real zeros and are compatible. Hence
$I(G-v;x)+4xI(G-N[v];x)\cos^2\frac{s\pi}{n+2}$ has only real zeros
for every $s$. Thus the product $I(\pngu;x)$ has only real zeros.
\end{proof}
\begin{rem}
Although the graphs $V_{n}^{(1)}$ and $V_{n}^{(2)}$ are not
claw-free for $n\ge 2$, their independence polynomials have only
real zeros from Proposition~\ref{cl}. As a contrast,
$I(V_{n}^{(m)};x)$ is merely log-concave and unimodal in general for
$m\ge 3$.
\end{rem}
\subsubsection{Concatenation of Graphs in a Circle} \hspace*{\parindent}
Let $C_n$ denote the cycle with $n$ vertices. Let $G$ be a
graph and $v$ a vertex of $G$. Let $\cngu$ denote the graph obtained
by gluing each vertex of $C_n$ with the vertex $v$ of $n$ copies $G$
respectively. We have a similar result to Theorem~\ref{thm-PGu}.
\begin{prop}\label{prop-CGu}
Let $\cngu$ be defined as above. Then for $n\ge 3$,
$$I(\cngu;x)=I^{\lrf{\frac{n+1}{2}}}(G-v;x)
\prod_{s=1}^{\lrf{\frac{n}{2}}}
\left[I(G-v;x)+4xI(G-N[v];x)\cos^2\frac{(2s-1)\pi}{2n}\right].$$
\end{prop}
\begin{proof}
Let $f_n$ and $g_n$ be the independence polynomials of $\pngu$ and
$\cngu$ respectively. Then by Lemma~\ref{lem-IP}~(i),
\begin{equation}\label{f-g}
g_n=af_{n-1}+a^2bf_{n-3},\quad n=3,4,\ldots,
\end{equation}
where $a=I(G-v;x)$ and $b=xI(G-N[v];x)$. Recall that
$$f_n=\frac{{\la_1}^{n+2}-{\la_2}^{n+2}}{a(\la_1-\la_2)},$$
where $\la_1$ and $\la_2$ are the roots of the equation
$\la^2-a\la-ab=0$. Hence
\begin{eqnarray*}
g_n &=& a\frac{\la_1^{n+1}-\la_2^{n+1}}{a(\la_1-\la_2)}+a^2b\frac{\la_1^{n-1}-\la_2^{n-1}}{a(\la_1-\la_2)}\\
&=& \frac{\la_1^{n+1}-\la_2^{n+1}}{\la_1-\la_2}-\la_1\la_2\frac{\la_1^{n-1}-\la_2^{n-1}}{\la_1-\la_2}\\
&=& \la_1^n+\la_2^n\\
&=& a^{\lrf{\frac{n+1}{2}}}\prod_{s=1}^{\lrf{\frac{n}{2}}}
\left[a+4b\cos^2\frac{(2s-1)\pi}{2n}\right].
\end{eqnarray*}
by Lemma \ref{lem-decomp} (iii) and (iv). Thus the statement
follows.
\end{proof}
\begin{rem}
The sequence $f_n$ satisfies the linear recurrence relation
(\ref{rr-f}), so by (\ref{f-g}) the sequence $g_n$ satisfies the
same recursion:
$$g_n=ag_{n-1}+abg_{n-2}, \quad n=5,6,\ldots.$$
\end{rem}
\begin{rem}
If $G$ is the singleton graph with the unique vertex $v$, then
$\cngu$ is merely the cycle $C_n$. Thus the independence
polynomial is
$$I(C_n;x)=\prod_{s=1}^{\lrf{\frac{n}{2}}}
\left[1+4x\cos^2\frac{(2s-1)\pi}{2n}\right].$$
\end{rem}
\begin{rem}
There are similar results to Propositions~\ref{vg},
\ref{fg}~and~\ref{cl} for $\cngu$.
\end{rem}
\subsection{A Conjecture of Levit and Mandrescu}
\hspace*{\parindent}
In \cite{LM07}, Levit and Mandrescu constructed a family of graphs
$H_n$ from the path $P_n$ with $n$ vertices by the so-called ``clique cover of a
graph'' rule, as shown in Figure 2. By $H_0$ we mean the null graph.
\begin{center}
\setlength{\unitlength}{1cm}
\begin{picture}(20,2.5)(-1.5,-1.5)

\thicklines \put(0,0){\line(1,0){3.15}} \thicklines
\put(3.8,0){\line(1,0){2.2}}
\put(0,0){\circle*{0.2}}\put(1,0){\circle*{0.2}}
\put(2,0){\circle*{0.2}}\put(3,0){\circle*{0.2}}
\put(3.3,0){\circle*{0.1}}\put(3.5,0){\circle*{0.1}}
\put(3.7,0){\circle*{0.1}}
 \put(4,0){\circle*{0.2}}
\put(5,0){\circle*{0.2}}\put(6,0){\circle*{0.2}}
 \put(0,0){\line(0,1){1}} \put(0,0){\line(0,-1){1}}
 \put(1,0){\line(-1,-1){1}} \put(1,0){\line(-1,1){1}}
 \put(2,0){\line(0,-1){1}}\put(2,0){\line(0,1){1}}
 \put(3,0){\line(-1,-1){1}}\put(3,0){\line(-1,1){1}}
 \put(4,0){\line(0,-1){1}}\put(4,0){\line(0,1){1}}
 \put(5,0){\line(-1,-1){1}}\put(5,0){\line(-1,1){1}}
 \put(6,0){\line(0,-1){1}}\put(6,0){\line(0,1){1}}
 \put(0,-1){\circle*{0.2}}\put(0,1){\circle*{0.2}}
 \put(2,-1){\circle*{0.2}}\put(2,1){\circle*{0.2}}
\put(4,-1){\circle*{0.2}}\put(4,1){\circle*{0.2}}
\put(6,-1){\circle*{0.2}}\put(6,1){\circle*{0.2}}

\put(0.1,0.1){$v_{1}$}\put(1.5,-0.3){$e$}
\put(1.1,0.1){$v_{2}$}\put(2.1,0.1){$v_{3}$}
\put(5.1,0.1){$v_{2m}$}\put(6.1,0.1){$v_{2m+1}$}
\put(2,-1.8){$H_{2m+1}$}

\thicklines \put(8,0){\line(1,0){3.15}} \thicklines
\put(11.8,0){\line(1,0){1.2}}
\put(8,0){\circle*{0.2}} \put(9,0){\circle*{0.2}}
\put(10,0){\circle*{0.2}}\put(11,0){\circle*{0.2}}
\put(11.3,0){\circle*{0.1}}\put(11.5,0){\circle*{0.1}}
\put(11.7,0){\circle*{0.1}}
\put(12,0){\circle*{0.2}}\put(13,0){\circle*{0.2}}

\put(8,0){\line(1,-1){1}}\put(8,0){\line(1,1){1}}
 \put(9,0){\line(0,-1){1}}\put(9,0){\line(0,1){1}}
 \put(10,0){\line(1,-1){1}}\put(10,0){\line(1,1){1}}
 \put(11,0){\line(0,-1){1}}\put(11,0){\line(0,1){1}}
 \put(12,0){\line(1,-1){1}}\put(12,0){\line(1,1){1}}
 \put(13,0){\line(0,-1){1}}\put(13,0){\line(0,1){1}}

 \put(9,-1){\circle*{0.2}}\put(9,1){\circle*{0.2}}
\put(11,-1){\circle*{0.2}}\put(11,1){\circle*{0.2}}
\put(13,-1){\circle*{0.2}}\put(13,1){\circle*{0.2}}

\put(8.28,0.1){$v_1$}\put(9.1,0.1){$v_2$} \put(9.5,-0.3){$e$}
\put(10.28,0.1){$v_3$}\put(11.1,0.15){$v_4$}
\put(13.1,0.1){$v_{2m}$} \put(10,-1.8){$H_{2m}$}
\end{picture}\\[5mm]
Figure~2. The graph $H_n$.
\end{center}

Using Lemma~\ref{lem-IP} (i) Levit and Mandrescu~\cite{LM07} showed
that the independence polynomials of $H_n$ satisfy the recurrence
relation $$\left\{
  \begin{array}{ll}
     I(H_{2m};x)  =  I(H_{2m-1};x)+xI(H_{2m-2};x), & \hbox{} \\
    I(H_{2m+1};x)  =  (1+x)^2I(H_{2m};x)+xI(H_{2m-1};x), & \hbox{}
  \end{array}
\right.$$ with $I(H_0;x)=1$ and $I(H_1;x)=1+3x+x^2$. From this they
can manage to show that the independence polynomials of $H_n$ are
symmetric and unimodal. They further made the following conjecture.
\begin{conj}[{\cite{LM07}}]
\label{conj-LM} The independence polynomial of $H_n$ is log-concave
and has only real zeros for $n\ge 1$.
\end{conj}

We now use our approach to give the factorization of $I(H_n;x)$,
from which the result and conjecture of Levit and Mandrescu  are
clearly true.
\begin{thm}\label{thm-H}
Let $n\ge 1$. Then
\begin{itemize}
\item [\rm (i)]
the independence polynomial of $H_n$ is
\begin{equation}\label{exp-H}
I(H_{n};x)=\prod_{s=1}^{\lrf{(n+1)/2}}\left(1+4x+x^2+2x\cos\frac{2s\pi}{n+2}\right).
\end{equation}
\item [\rm (ii)]
$I(H_{n};x)$ is symmetric.
\item [\rm (iii)]
$I(H_{n};x)$ has only real zeros, and is therefore log-concave and
unimodal.
\end{itemize}
\end{thm}
\begin{proof}
Let $h_n=I(H_{n};x)$. Then by Lemma~\ref{lem-IP} (ii) we obtain
\begin{equation}\label{rr-h}
h_{n+2}=(1+4x+x^2)h_{n}-x^2h_{n-2}
\end{equation}
for $n\ge 2$, where $$h_0=1,
h_1=1+3x+x^2,h_2=1+4x+x^2,h_3=1+7x+13x^2+7x^3+x^4.$$ For
convenience, we set $h_{-1}=1$, which is well-defined extension by
(\ref{rr-h}).

We next use Lemma \ref{lem-BF} to deduce the factorization of $h_n$
from the recursion (\ref{rr-h}). We may do this only for $x>0$ since
every $h_n$ is a polynomial in $x$. Let $\la_1$ and $\la_2$ be the
roots of quadratic equation
$$\la^2-(1+4x+x^2)\la+x^2=0.$$
Then $\la_1+\la_2=1+4x+x^2$ and $\la_1\la_2=x^2$.

First consider the sequence $\{h_{2m}\}$. By Lemma~\ref{lem-BF} we
have
$$h_{2m}=\frac{\left[\left(1+4x+x^2\right)-\la_2\right]\la_1^m+\left[\la_1-\left(1+4x+x^2\right)\right]\la_2^m}{\la_1-\la_2}
=\frac{\la_1^{m+1}-\la_2^{m+1}}{\la_1-\la_2}.$$ It follows from
(\ref{d-decomp}) that
\begin{equation}\label{even}
h_{2m}=\prod_{s=1}^{m}\left(1+4x+x^2+2x\cos\frac{s\pi}{m+1}\right).
\end{equation}

Then consider the sequence $\{h_{2m-1}\}$. Note that
$1+3x+x^2=(1+4x+x^2)-x=\la_1+\la_2-\sqrt{\la_1\la_2}$. Hence
\begin{eqnarray*}
  h_{2m-1} &=& \frac{\left[\left(1+3x+x^2\right)-\la_2\right]\la_1^m+\left[\la_1-\left(1+3x+x^2\right)\right]\la_2^m}{\la_1-\la_2} \\
   &=& \frac{\left(\la_1-\sqrt{\la_1\la_2}\right)\la_1^m+\left(\sqrt{\la_1\la_2}-\la_2\right)\la_2^m}{\la_1-\la_2} \\
   &=&
   \frac{\left(\sqrt{\la_1}\right)^{2m+1}+\left(\sqrt{\la_2}\right)^{2m+1}}{\sqrt{\la_1}+\sqrt{\la_2}}.
\end{eqnarray*}
It follows from Lemma~\ref{lem-decomp}~(iii) that
\begin{eqnarray}\label{odd}
h_{2m-1}&=&\prod_{s=1}^{m}\left[\left(\sqrt{\la_1}+\sqrt{\la_2}\right)^2-4\sqrt{\la_1\la_2}\cos^2\frac{(2s-1)\pi}{2(2m+1)}\right]\nonumber\\
&=&\prod_{s=1}^{m}\left(\la_1+\la_2-2\sqrt{\la_1\la_2}\cos\frac{(2s-1)\pi}{2m+1}\right)\nonumber\\
&=&\prod_{s=1}^{m}\left(1+4x+x^2-2x\cos\frac{(2s-1)\pi}{2m+1}\right)\nonumber\\
&=&\prod_{s=1}^{m}\left(1+4x+x^2+2x\cos\frac{2s\pi}{2m+1}\right).
\end{eqnarray} Combine (\ref{even}) and (\ref{odd}) to obtain
(\ref{exp-H}). This proves (i).

Clearly, every factor on the right of (\ref{exp-H})
$$1+4x+x^2+2x\cos\frac{2s\pi}{n+2}=1+2\left(2+\cos\frac{2s\pi}{n+2}\right)x+x^2$$
is symmetric and has only real zeros, so the product $I(H_n;x)$ has
the same property. Thus (ii) and (iii) follow, and the proof of the
theorem is therefore completed.
\end{proof}
\begin{rem}
Levit and Mandrescu~\cite{LM07} verified that all zeros of
$I(H_n;x)$ are located in the interval $(-6,0)$ for $n\le 20$ and
conjectured that it is true in general. Now we can prove this
conjecture by means of the factorization (\ref{exp-H}) of
$I(H_n;x)$. Indeed, every zero of $I(H_n;x)$ is obviously negative
and its absolute value
$$\left|-\left(2+\cos\frac{2s\pi}{n+2}\right)\pm\sqrt{\left(2+\cos\frac{2s\pi}{n+2}\right)^2-1}\right|<2\left(2+\cos\frac{2s\pi}{n+2}\right)\le 6.$$
\end{rem}
\section{Concluding Remarks}
\hspace*{\parindent}
In this paper we have developed techniques for attacking unimodality
problems of the independence polynomial of a graph. Roughly
speaking, one can use Lemma~\ref{lem-IP} to obtain the recurrence
relation for independence polynomials for certain class of graphs.
Solve this recursion by Lemma~\ref{lem-BF}. It is possible to give
the factorization for such polynomials by Lemma~\ref{lem-decomp}.
Then the unimodality of the product follows from that of its factors
by Lemma~\ref{product}.

The key step in this approach is to give the factorization of
polynomials $Q_n(x)$ satisfying a linear recurrence relation
$$Q_n(x)=a(x)Q_{n-1}(x)+b(x)Q_{n-2}(x).$$
By Lemma~\ref{lem-BF} we have
$$Q_n(x)=\frac{[Q_1(x)-\la_2Q_0(x)]\la_1^n+[\la_1Q_0(x)-Q_1(x)]\la_2^n}{\la_1-\la_2},$$
where $\la_1$ and $\la_2$ are the roots of quadric equation
$\la^2-a(x)\la-b(x)=0$. The reader may wonder when $Q_n(x)$ can be
factorized by means of Lemma~\ref{lem-decomp}. The solution
obviously depends on the initial values $Q_0(x)$ and $Q_1(x)$. For
example, if $Q_1(x)=a(x)Q_0(x)$, then
$$Q_n(x)=Q_0(x)\frac{\la_1^{n+1}-\la_2^{n+1}}{\la_1-\la_2}
=Q_0(x)a(x)^{\delta_n}\prod_{s=1}^{\lrf{n/2}}\left[a^2(x)+4b(x)\cos^2\frac{s\pi}{n+1}\right],\quad
n\ge 0$$ by Lemma \ref{lem-decomp}, where $\delta_n=0$ for even $n$
and $1$ otherwise. Similarly, if $Q_1(x)=a(x)Q_0(x)/2$, then
$$Q_n(x)=\frac{1}{2}Q_0(x)\left(\la_1^{n}+\la_2^{n}\right)
=\frac{1}{2}Q_0(x)a(x)^{\delta_n}\prod_{s=1}^{\lrf{n/2}}\left[a^2(x)+4b(x)\cos^2\frac{(2s-1)\pi}{2n}\right],\quad
n\ge 0.$$

It is worth pointing out that our approach to independence
polynomials may be adapted to other graph polynomials. We take the
adjoint polynomial of a graph as an example. Given a graph
$G=(V,E)$, let $\alpha(G,k)$ be the number of partitions of $V$ into
$k$ non-empty independent sets. The adjoint polynomial of $G$ is
defined by
$$h(G;x)=\sum_{k=1}^{n}\alpha(\overline{G},k)x^k,$$
where $n$ is the order of $G$ and $\overline{G}$ is the complement
of $G$. The notion of the adjoint polynomials was introduced by
Liu~\cite{Liu87} to study the chromatic uniqueness of graphs.
Clearly, two graphs have the same chromatic polynomials if and only
if their complements have the same adjoint polynomials. There are
some similar equalities for adjoint polynomials to those for
independence polynomials occurred in Lemma~\ref{lem-IP}. For
example, let $e=uv\in E(G)$. If $e$ is not an edge of any triangle
of $G$, then
$$h(G;x)=h(G-e;x)+xh(G-\{u,v\};x).$$
It follows that the recurrence relation
\begin{equation}\label{h}
h(G_n;x)=xh(G_{n-1};x)+xh(G_{n-2};x)
\end{equation}
holds for certain classes of graphs $G_n$, including the path
$P_n$ and the cycle $C_n$ (see~\cite{DTLH02} for details).
Now set $h(P_0;x):=1, h(P_1;x)=x$ and $h(C_0;x):=2, h(C_1;x):=x$,
which are well-defined extension by (\ref{h}). Then we have
$$h(P_n;x)=x^{\lrf{(n+1)/2}}\prod_{s=1}^{\lrf{n/2}}\left(x+4\cos^2\frac{s\pi}{n+1}\right)$$
and
$$h(C_n;x)=x^{\lrf{(n+1)/2}}\prod_{s=1}^{\lrf{n/2}}\left(x+4\cos^2\frac{(2s-1)\pi}{2n}\right).$$
Dong et al.~\cite{DTLH02} have obtained these two factorizations by
a somewhat different method, which is one of the motivations of the
present paper.

It often occurs that the unimodality of a polynomial is known, yet
to determine the exact number and location of modes is a much more
difficult task. The case for polynomials with only real zeros is
somewhat different. Let $Q(x)=\sum_{k=0}^{n}a_kx^k$ be a polynomial
with positive coefficients. Darroch~\cite{Dar64} showed that if all
zeros $r_1,\ldots,r_n$ of $Q(x)$ are real, then $Q(x)$ has at most
two modes and each mode $m$ satisfies $|m-M|<1$, where
$$M:=\frac{Q'(1)}{Q(1)}=\frac{\sum_{k=0}^nka_k}{\sum_{k=0}^na_k}=\sum_{k=1}^n\frac{1}{1-r_k}.$$
In \cite{LM02}, Levit and Mandrescu conjectured that the mode of the
independence polynomial of the centipede $V_n^{(1)}$ is $n-f(n)$,
where $f(n)$ is given by
$$f(n)=\left\{
  \begin{array}{ll}
    1+\lrf{n/5}, & \hbox{if $2\le n\le 6$;} \\
    f(2+(n-2)\bmod 5)+2\lrf{(n-2)/5}, & \hbox{if $n\ge 7$.}
  \end{array}
\right.$$ We can apply Darroch's result to give a counterexample.
Let us examine $I(V_{142}^{(1)};x)$. We have by (\ref{exp-v})
$$I(V_{142}^{(1)};x)=(1+x)^{71}\prod_{s=1}^{71}\left[1+x\left(1+4\cos^2\frac{s\pi}{144}\right)\right].$$
Thus
$$M=\sum_{k=1}^{142}\frac{1}{1-r_k}=\frac{213}{2}-\sum_{s=1}^{71}\frac{1}{2+4\cos^2\frac{s\pi}{144}}\approx 86.0487$$
by means of the computer program MATHEMATICA 5.2. However,
$$142-f(142)=142-57=85,$$
which is not a mode of $I(V_{142}^{(1)};x)$ by Darroch's result.
This gives a negative answer to the conjecture of Levit and
Mandrescu. Actually, the exact (unique) mode of $I(V_{142}^{(1)};x)$
is $86$ again by MATHEMATICA 5.2:
$$I(V_{142}^{(1)};x)\approx\cdots+7.18929\times 10^{60}x^{85}+7.33386\times 10^{60}x^{86}+7.24852\times
10^{60}x^{87}+\cdots.$$

Finally, we refer the reader to
\cite{Bra06,Bre89,Bre92,Bre94,BRW94,MW08,Sta89,WYjcta05,WYeujc05,WYjcta07}
for more results about unimodality problems of sequences and
polynomials.
\section*{Acknowledgements}
\hspace*{\parindent}
The authors thank the anonymous referees for their helpful comments and suggestions.


\begin{thebibliography}{23}
\bibitem{AMSE87}
Y. Alavi, P.J. Malde, A.J. Schwenk, P. Erd\H{o}s, The vertex
independence sequence of a graph is not constrained, Congr. Numer.
58 (1987) 15--23.
\bibitem{BC55}
S. Barnard, J.F. Child, Higher Algebra, Macmillan, London, 1955.
\bibitem{BM76}
J.A. Bondy, U.S.R. Murty, Graph Theory with Applications, Macmillan
Press, New York, 1976.
\bibitem{Bra06}
P. Br\"and\'en, On linear transformations preserving the P\'olya
frequency property, Trans. Amer. Math. Soc. 358 (2006) 3697--3716.
\bibitem{Bre89}
F. Brenti, Unimodal, log-concave, and P\'olya frequency sequences in
combinatorics,  Mem. Amer. Math. Soc. 413 (1989).
\bibitem{Bre92}
F. Brenti, Expansions of chromatic polynomials and log-concavity,
Trans. Amer. Math. Soc. 332 (1992) 729--756.
\bibitem{Bre94}
F. Brenti, Log-concave and unimodal sequences in algebra,
combinatorics, and geometry: An update, Contemp. Math. 178 (1994)
417--441.
\bibitem{BRW94}
F. Brenti, G. Royle, D.G. Wagner, Location of zeros of chromatic and
related polynomials of graphs, Canad. J. Math. 46 (1994) 55--80.
\bibitem{BDN00}
J.I. Brown, K. Dilcher, R.J. Nowakowski,
Roots of independence polynomials of well-covered graphs,
J. Algebr. Combin. 11 (2000) 197--210.
\bibitem{BHN04}
J.I. Brown, C.A. Hickman, R.J. Nowakowski, On the location of roots
of independence polynomials, J. Algebraic Combin. 19 (2004) 273--282.
\bibitem{BN05}
J.I. Brown, R.J. Nowakowski, Average independence polynomials, J.
Combin. Theory Ser. B 93 (2005) 313--318.
\bibitem{Bru92}
R.A. Brualdi, Introductory combinatorics, 2nd Ed., North-Holland,
Amsterdam, 1992.
\bibitem{CLY97}
W.-Y Chen, H.-I. Lu, Y.-N. Yeh, Operations of Interlaced Trees and
Graceful Trees, Southeast Asian Bull. Math. 21 (1997) 337--348.
\bibitem{CS07}
M. Chudnovsky, P. Seymour, The roots of the independence polynomial
of a clawfree graph, J. Combin. Theory Ser. B 97 (2007) 350--357.
\bibitem{Dar64}
J.N. Darroch, On the distribution of the number of successes in
independent trials, Ann. Math. Statist. 35 (1964) 1317--1321.
\bibitem{DTLH02}
F.M. Dong, K.L. Teo, C.H.C. Little, M.D. Hendy, Zeros of adjoint
polynomials of paths and cycles, Australas. J. Combin. 25 (2002)
167--174.
\bibitem{GJ79}
M.R. Garey, D.S. Johnson, Computers and Intractability: A Guide to
the Theory of NP-Completeness, W.H. Freeman and Company, New York,
1979.
\bibitem{GH83}
I. Gutman, F. Harary, Generalizations of the matching polynomial,
Utilitas  Math. 24 (1983) 97--106.
\bibitem{HLP52}
G.H. Hardy, J.E. Littlewood, G. P\'olya, Inequalities, Cambridge
University Press, Cambridge, 1952.
\bibitem{Ham90}
Y.O. Hamidoune, On the number of independent k-sets in a claw-free
graph, J. Combin. Theory Ser. 50 (1990) 241--244.
\bibitem{HL72}
O.J. Heilmann, E.H. Lieb, Theory of monomer-dimer systems, Comm.
Math. Phys. 25 (1972) 190--232.
\bibitem{HL94}
C. Hoede, X.L. Li, Clique polynomials and independent set
polynomials of graphs, Discrete Math. 125 (1994) 219--228.
\bibitem{LM02}
V.E. Levit, E. Mandrescu, On well-covered trees with unimodal
independence polynomials, Congr. Numer. 159 (2002) 193--202.
\bibitem{LM04CJM}
V.E. Levit, E. Mandrescu,
Very well-covered graphs with log-concave independence polynomials,
Carpathian J. Math. 20 (2004) 73--80.
\bibitem{LM04WSEAS}
V.E. Levit, E. Mandrescu,
Graph products with log-concave independence polynomials,
WSEAS Trans. Math. 3 (2004) 487--492.
\bibitem{LM05}
V.E. Levit, E. Mandrescu, The independence polynomial of a graph
--- a survey, Proceedings of the 1st International Conference on
Algebraic Informatics, 233--254, Aristotle Univ. Thessaloniki,
Thessaloniki, 2005.
\bibitem{LM06CN}
V.E. Levit, E. Mandrescu,
Partial unimodality for independence polynomials of K\"{o}nig-Egerv\'{a}ry graphs,
Congr. Numer. 179 (2006) 109--119.
\bibitem{LM06EJC}
V.E. Levit, E. Mandrescu,
Independence polynomials of well-covered graphs: generic counterexamples for the unimodality conjecture,
European J. Combin. 27 (2006) 931--939.
\bibitem{LM03}
V.E. Levit, E. Mandrescu, Independence polynomials and the
unimodality conjecture for very well-covered, quasi-regularizable,
and perfect graphs, Graph theory in Paris, 243--254, Trends Math.,
Birkh\"auser, Basel, 2007.
\bibitem{LM07}
V.E. Levit, E. Mandrescu, A family of graphs whose independence
polynomials are both palindromic and unimodal, Carpathian J. Math.
23 (2007) 108--116.
\bibitem{LW07}
L.L. Liu, Y. Wang, A unified approach to polynomia sequences with
only real zeros, Adv. in. Appl. Math. 35 (2007) 542--560.
\bibitem{Liu87}
R.Y. Liu, On chromatic polynomials of two classes of graphs, Kexue
Tongbao 32 (1987) 1147--1148 (in Chinese).
\bibitem{MW08}
S.-M. Ma and Y. Wang, $q$-Eulerian polynomials and polynomials with
only real zeros, Electron. J. Combin. 15 (2008), Research Paper 17,
9 pp.
\bibitem{Man09}
E. Mandrescu, Building graphs whose independence polynomials have only real roots,
Graphs Combin. 25 (2009) 545--556.
\bibitem{Rea68}
R.C. Read, An introduction to chromatic polynomials, J. Combin.
Theory 4 (1968) 52--71.
\bibitem{Sch81}
A.J. Schwenk, On unimodal sequence of graphical invariants. J.
Combin. Theory Ser. B 30 (1981) 247--250.
\bibitem{Sta89}
R.P. Stanley, Log-concave and unimodal sequences in algebra,
combinatorics and geometry, Ann. New York Acad. Sci. 576 (1989)
500--534.
\bibitem{WYjcta05}
Y. Wang and Y.-N. Yeh, Polynomials with real zeros and P\'olya
frequency sequences, J. Combin. Theory Ser. A 109 (2005) 63--74.
\bibitem{WYeujc05}
Y. Wang and Y.-N. Yeh, Proof of a conjecture on unimodality,
European J. Combin. 26 (2005) 617--627.
\bibitem{WYjcta07}
Y. Wang and Y.-N. Yeh, Log-concavity and LC-positivity, J. Combin.
Theory Ser. A 114 (2007) 195--210.
\bibitem{Wel76}
D. Welsh, Matriod Theory, Academic Press, London/New York, 1976.
\bibitem{Zhu07}
Z.F. Zhu, The unimodality of independence polynomials of some
graphs, Australas. J. Combin. 38 (2007) 27--33.
\end{thebibliography}
\end{document}